\documentclass[12pt]{amsart}
\usepackage{amsmath,amssymb,euscript,mathrsfs,bm,enumerate,pstricks}

\usepackage{multido}
\usepackage{pst-plot}
\usepackage{pxfonts}

\usepackage{setspace}
\usepackage[shortlabels]{enumitem}

\usepackage{tikz-cd}
\usepackage{graphicx}
\graphicspath{ {./images/} }

\usepackage[colorlinks=true,linkcolor=black,citecolor=black,urlcolor=black]{hyperref}

\usepackage{subcaption}
\pushQED{\qed}

\newcommand{\excise}[1]{}

\renewcommand{\setminus}{\smallsetminus}
\renewcommand{\phi}{\varphi}

\newcommand{\E}{\mathbb{E}}

\renewcommand{\P}{\mathbb{P}}

\newcommand{\R}{\mathbb{R}}

\newcommand{\Z}{\mathbb{Z}}

\newcommand{\ba}{\textbf{a}}

\newtheorem{theorem}{Theorem}[section] 
\newtheorem{lemma}[theorem]{Lemma}
\newtheorem{proposition}[theorem]{Proposition}
\newtheorem{corollary}[theorem]{Corollary}

\newtheorem*{thm*}{Theorem}     
\newtheorem*{cor*}{Corollary}
\newtheorem*{prop*}{Proposition}
\newtheorem*{lem*}{Lemma}

\theoremstyle{definition}       
\newtheorem{definition}[theorem]{Definition}
\newtheorem{remark}[theorem]{Remark}
\newtheorem{example}[theorem]{Example}

\newtheorem*{rmk*}{Remark}   
\newtheorem*{question*}{Question}
\newtheorem*{claim*}{Claim}

\numberwithin{theorem}{section}

\begin{document}


\title[Limits of polyhedral multinomial distributions]{Limits of polyhedral multinomial distributions}
\author{Aniket Shah}
\email{shah@karlin.mff.cuni.cz}
\address{Department of Algebra, Faculty of Mathematics and Physics, Charles University, Prague, Czech Republic}
\thanks{The author was supported by Charles University project PRIMUS/21/SCI/014.}
\date{\today}

\begin{abstract}
We consider limits of certain measures supported on lattice points in lattice polyhedra defined as the intersection of half-spaces $\{m\in\R^n| \langle v_i, x\rangle +a_i\geq 0\}$, where $\sum_i v_i =0$. The measures are densities associated to lattice random variables obtained by restriction of multinomial random variables. We find the limiting Gaussian distributions explicitly.
\end{abstract}

\maketitle

\section{Introduction}
Let $\ba = (a_1,\ldots, a_r)\in\Z^r$ and $v_1,\ldots,v_r\in \Z^n$ define hyperplanes $H_i = \{x\in \R^n | \langle v_i, x \rangle +a_i = 0\}$, such that the corresponding intersection of half-spaces $P=\bigcap_i \{x\in\R^n | \langle v_i, x \rangle +a_i \geq 0\}$ is compact, and each $H_i$ touches $P$. Then, we define the random vector $\textbf{X}_{\ba}$ in $\R^n$ by 
\[
\P(\textbf{X}_{\textbf{a}}= x) =  \frac{1}{b}\cdot \binom{\sum_i \langle v_i, x\rangle + a_i }{\langle v_1, x\rangle + a_1, \ldots, \langle v_r, x\rangle + a_r},
\]
for each lattice point $x\in P\cap\Z^n$. We call its distribution the \textit{polyhedral multinomial distribution} associated to $\ba$. 

Imposing the further condition that $\sum_i v_i = 0$, the top term in the multinomial simplifies to $|\ba| := \sum_i a_i$, so 
\[
\P(\textbf{X}_{\ba}= x) =  \frac{1}{b}\cdot \binom{|\ba| }{\langle v_1, x\rangle + a_1, \ldots, \langle v_r, x\rangle + a_r}.
\]
In the special case that $r=n+1$, $v_i=e_i$ for $i$ from $1$ to $n$, and $v_{n+1}=-\sum_i e_i$, the corresponding distribution is the usual multinomial. More generally, $\textbf{X}_{\ba}$ follows a conditional distribution of a multinomial distribution on a higher dimensional vector space. 

We are interested in the limiting behavior of $\textbf{X}_{k\cdot\ba}$ as $k$ goes to infinity. For multinomial distributions, the limit is well-known to approach a Gaussian, due to the central limit theorem. More generally, the distribution of $\textbf{X}_{k\cdot\ba}$ does not remain within the the region (which grows sublinearly in $k$ around the mean) controlled by e.g. the central or local limit theorems \cite{durrett, ouimet} for higher dimensional multinomial distributions. Our approach is to instead approximate directly using Stirling's formula.

Our main result is that after recentering and scaling appropriately, the polyhedral multinomial distributions $\textbf{X}_{k\cdot\ba}$ converge to a specific Gaussian distribution supported on the subspace $L$ generated by differences of vectors in $P$. We let $I_{\ba}\subset\{1,\ldots,r\}$ denote the indices $i$ such that $\langle v_i,x\rangle$ is not constant on $P$. The point $m_\ba\in P$ is defined in Proposition \ref{prop.mdag}.
\begin{theorem}
When $\sum_i v_i = 0$, the sequence $\textbf{Y}_k = \frac{ \textbf{X}_{k\cdot\ba}-\E[\textbf{X}_{k\cdot\ba}]}{\sqrt{k}}$ converges weakly to $\textbf{Y}$ with density given by the Dirac measure $\delta_L$ times a Gaussian with mean $0$ and variance given by the quadratic form 
\[
x\mapsto \sum_{i\in I_{\ba}} \frac{\left(\langle v_i, x\rangle\right)^2}{\langle v_i,m_{\ba}\rangle+a_i}.
\]
\end{theorem}

%
%

\begin{remark}
This research was motivated by consideration of an analogue of Duistermaat-Heckman measure in \cite{anderson-shah} while studying line bundles on the toric arc scheme as defined in \cite{ak}. Those familiar with toric geometry may recognize that besides $\sum_i v_i=0$, the conditions on $v_i$ and $\ba$ relate to $\sum_i a_i D_i$ defining a nef divisor on an associated toric variety.
\end{remark}

\section{The potential corresponding to the data $(a_1,\ldots,a_r)$}

Let $\delta_{x}$ be the Dirac probability measure supported at $x\in\R^n$. We let $\ba = (a_1,\ldots, a_r)\in\Z^r$ and $v_1,\ldots,v_r\in \Z^n$ define $H_i = \{x\in \R^n | \langle v_i, x \rangle +a_i = 0\}$ as in the introduction. We assume the intersection $P=\bigcap_i \{x\in\R^n | \langle v_i, x \rangle +a_i \geq 0\}$ is compact, and each $H_i$ touches $P$. The distribution of the random vector $\textbf{X}_{\ba}$ is
\[
\mu_{\ba} = \frac{1}{b}\cdot\sum_{x\in P\cap\Z^n}\binom{|\ba|}{\langle v_1, x\rangle + a_1, \ldots, \langle v_r, x\rangle + a_r}\delta_x.
\]
and the distribution for $\textbf{Y}_k = \frac{ \textbf{X}_{k\cdot\ba}-\E[\textbf{X}_{k\cdot\ba}]}{\sqrt{k}}$ is
\[
\nu_k := \left(\tau_{\sqrt k}\right)_*\left( \mu_{k\cdot\ba}\ast \delta_{-\E[\textbf{X}_{k\cdot\ba}]}\right).
\]

We will relate $\nu_{k}$ to the following function, which we call the \textit{potential} of \ba. 
\begin{definition}\label{def.phi}
Let $\phi_{\ba}: P \rightarrow \R_{> 0}$ be the function
\[
x\mapsto \prod_{i=1}^r \left(\langle v_{i}, x \rangle+a_{i}\right)^{\langle v_{i}, x \rangle+a_{i}} .
\]
In this product we read $0^0$ as $1$, so $\prod_{i=1}^r \left(\langle v_{i}, x \rangle+a_{i}\right)^{\langle v_{i}, x \rangle+a_{i}} = \prod_{i\in I_\ba} \left(\langle v_{i}, x \rangle+a_{i}\right)^{\langle v_{i}, x \rangle+a_{i}} .$
\end{definition}

We let $\textrm{rel.int.}(P)$ be the interior of $P$ viewed as a subspace of the affine linear span of vectors in $P$. For $x\in \textrm{rel.int.}(P)$ and each $i\in I_\ba$, $\langle v_i, x\rangle + a_i > 0$.

\begin{proposition}\label{prop.mdag}
The function $\phi_{\ba}$ is convex on $P$, and there is a unique $m_{\ba}\in\textrm{rel.int.}(P)$ minimizing $\phi_{\ba}$.
\end{proposition}
\begin{proof}
Let $x$ be in the relative interior of $P$, so $\langle v_i, x\rangle + a_i > 0$ for $i\in I_\ba$, and  let $x'$ be in $L$. We calculate
\[
\frac{d \left(\phi_{\ba}(x+tx')\right) }{dt} = \phi_{\ba}(x)\cdot\sum_{i\in I_{\ba}}\langle  v_i, x'\rangle\left(\log(\langle v_i, x+tx' \rangle+a_i)+1\right).
\]

We have assumed that $0=\sum_i v_i =\sum_{i\in I_\ba}v_i + \sum_{i\not\in I_\ba} v_i$. Any $x'\in L$ can be written as a sum of differences of elements of $P$, so for $i\not\in I_\ba$, we have that $\langle v_i, x'\rangle =0$. Then we can calculate that $\sum_{i\in I_{\ba}} \langle v_i, x'\rangle = 0$ as well.
Thus,
\[
\frac{d \left(\phi_{\ba}(x+tx')\right) }{dt} 
=
 \phi_{\ba}(x)\cdot \left(\sum_{i\in I_{\ba}}\langle v_i, x' \rangle\log(\langle v_i, x+tx' \rangle+a_i)\right).
\]

If $x$ is in the relative interior of $P$, $\langle v_i, x \rangle + a_i>0$ for $i\in I_\ba$. As $x+tx'$ approaches the boundary of $P$ at some finite positive $t_0$, we have that for some $i\in I_\ba$, $\langle v_i, x + tx'\rangle + a_i$ decreases to $0$. This implies both that $\langle v_i, x' \rangle <0$, and that $\log(\langle v_i, x+tx' \rangle+a_i)$ goes to negative infinity.

Thus, $\frac{d \left(\phi_{\ba}(x+tx')\right) }{dt}$ is positive near the boundary, so $\phi_{\ba}(x+tx')$ cannot be minimized there.

On the other hand, we see that the second derivative
\[
\frac{d^2 \left(\phi_{\ba}(x+tx')\right) }{dt^2} = \phi_{\ba}(x)\cdot\left(\left(\sum_{i\in I_{\ba}}\langle  v_i, x'\rangle\log(\langle v_i, x+tx' \rangle+a_i)\right)^2 + \sum_i \frac{\langle v_i, x' \rangle^2}{\langle v_i, x+t x'\rangle + a_i}  \right),
\]
is strictly positive when $x+tx'\in rel.int.(P)$, so $\phi_{\ba}(x)$ is convex. 

Thus, there is a unique minimizer $m_{\ba}$ in the relative interior of $P$.
\end{proof}

\begin{example}
Let $v_1=1$, $v_2 = -1$ in $\R$, and let $\ba = (a_1,a_2) = (0,l)$. Then $P$ is the interval $[0,l]\subset\R$, and $\phi_{\ba}$ is
\[
\phi_{\ba}(x) = x^x(l-x)^{l-x},
\]
which is minimized at $\frac{l}{2}$.

\end{example}

We now show the following technical lemma for $\phi_{\ba}$, which we will use later. Note that for any $x\in L$, and $k$ large enough, $\phi_{k\cdot\ba}(km_\ba+\sqrt{k} x)$ is defined.
\begin{lemma}\label{lem.maintech}
For all $x\in L$,
\[
\lim_{k\rightarrow\infty}  \frac{\phi_{k\cdot \ba}(km_\ba)}{\phi_{k\cdot \ba}(km_\ba+\sqrt{k}x)} = e^{-\frac{1}{2}\sum_{i\in I_\ba}\frac{\langle v_i, x \rangle^2}{\langle m_\ba, v_i\rangle +a_i}}.
\]

We show this by showing a fortiori that if $x\in L$ and $|x|<k^{c}$ for some $0<c<\frac{1}{6}$, then
\[
\left|
\frac{\phi_{k\cdot \ba}(km_\ba)}{\phi_{k\cdot \ba}(km_\ba+\sqrt{k}x)}
\sqrt{\prod_{i\in I_{\ba}}\frac{ \langle v_i,km_{\ba} \rangle + k a_i  }{\langle v_i,km_{\ba} + \sqrt{k}x \rangle + k a_i  }}
-
e^{-\frac{1}{2}\sum_{i\in I_{\ba}} \frac{\langle v_i,x\rangle^2}{\langle v_i,m_{\ba}\rangle + a_i}}
\right|
=e^{-\frac{1}{2}\sum_{i\in I_{\ba}} \frac{\langle v_i,x\rangle^2}{\langle v_i,m_{\ba}\rangle + a_i}}\cdot O(k^{3c-\frac{1}{2}}).
\]

\end{lemma}
\begin{proof}
We can rewrite the $k$-dependent expression as a product of three parts which we will deal with separately:
\begin{align}
\frac{\phi_{k\cdot \ba}(km_\ba)}{\phi_{k\cdot \ba}(km_\ba+\sqrt{k}x)}
\sqrt{\prod_{i\in I_{\ba}}\frac{ \langle v_i,km_{\ba} \rangle + k a_i  }{\langle v_i,km_{\ba} + \sqrt{k}x \rangle + k a_i  }}
= & 
\prod_{i\in I_{\ba}}\sqrt{\frac{\langle v_i,m_{\ba}\rangle + a_i}{\langle v_i,m_{\ba} + \frac{x}{\sqrt{k}} \rangle +  a_i}}\\
\cdot &
\left(\prod_{i\in I_{\ba}} \frac{\langle v_i,m_{\ba}\rangle + a_i}{\langle v_i,m_{\ba} + \frac{x}{\sqrt{k}} \rangle + a_i} \right)^{k(\langle v_i,m_{\ba} \rangle + a_i)}\\
\cdot &
\prod_{i\in I_{\ba}} \frac{1}{ \left(\langle v_i,m_{\ba} + \frac{x}{\sqrt{k}} \rangle + a_i\right)^{\sqrt{k}\langle v_i,x\rangle}}.
\end{align}

If $x\in L$ and $|x|<k^c$, then
\[
\sqrt{\frac{\langle v_i,m_{\ba}\rangle + a_i}{\langle v_i,m_{\ba} + \frac{x}{\sqrt{k}} \rangle +  a_i}}=1+O(k^{c-\frac{1}{2}}).
\]
For the second, which can be written $\prod_{i\in I_{\ba}} \left(\frac{1}{1+\frac{\langle v_i,x\rangle}{\sqrt{k}(\langle v_i,m_{\ba}\rangle+a_i)}} \right)^{k(\langle v_i,m_{\ba}\rangle+a_i)}$ we have
\[
\prod_{i\in I_{\ba}} \left(\frac{1}{1+\frac{\langle v_i,x\rangle}{\sqrt{k}(\langle v_i,m_{\ba}\rangle+a_i)}} \right)^{k(\langle v_i,m_{\ba}\rangle+a_i)}
=
\exp\left(-\sum_{i\in I_{\ba}} k(\langle v_i,m_{\ba}\rangle+a_i)\sum_{l=1}^\infty\frac{(-1)^{l-1}}{l}\left(\frac{\langle v_i,x\rangle}{\sqrt{k}(\langle v_i,m_{\ba}\rangle+a_i)}\right)^{l}\right).
\]
The argument of $\exp$ can be written
\[
-\sqrt{k}\left(\sum_{i\in I_{\ba}} \langle v_i,x\rangle\right) + \left(\sum_{i\in I_{\ba}}\frac{1}{2}\frac{\langle v_i,x\rangle^2}{\langle v_i,m_{\ba}\rangle+a_i}\right)-\frac{1}{\sqrt k}\left(\sum_{i\in I_{\ba}} \frac{\langle v_i,x\rangle^3}{(\langle v_i,m_{\ba}\rangle+a_i)^2}\sum_{l=0}^\infty\frac{1}{l+3}\left(\frac{-\langle v_i,x\rangle}{\sqrt{k}(\langle v_i,m_{\ba}\rangle+a_i)}\right)^l \right).
\]
Since $x\in L$, $\sum_{i\in I_{\ba}} \langle v_i,x\rangle=0$. If $|x|<k^c$, then the above is
\[
\left(\sum_{i\in I_\ba}\frac{1}{2}\frac{\langle v_i,x\rangle^2}{\langle v_i,m_{\ba}\rangle+a_i}\right)+O(k^{3c-\frac{1}{2}}),
\]
so  
\[
\left(\prod_{i\in I_\ba} \frac{\langle v_i,m_{\ba}\rangle + a_i}{\langle v_i,m_{\ba} + \frac{x}{\sqrt{k}} \rangle + a_i} \right)^{k(\langle v_i,m_{\ba} \rangle + a_i)}
=
e^{\frac{1}{2}\sum_{i\in I_\ba}\frac{\langle v_i,x\rangle^2}{\langle v_i,m_{\ba}\rangle+a{\rho_i}}}\cdot(1+O(k^{3c-\frac{1}{2}})).
\]

Finally, the last term in the product can be factored further:
\[
\prod_i \frac{1}{ \left(\langle v_i,m_{\ba} + \frac{x}{\sqrt{k}} \rangle + a_i\right)^{\sqrt{k}\langle v_i,x\rangle}}
=
\prod_i \left(\frac{1}{ \left(\langle v_i,m_{\ba}\rangle + a_i\right)^{\langle v_i,x\rangle}}\right)^{\sqrt k}
 \cdot
 \left(\frac{1}{1+\frac{\langle v_i,x\rangle}{\sqrt{k}(\langle v_i,m_{\ba}\rangle+a_i)}} \right)^{\sqrt{k}\langle v_i,x\rangle}.
\]
Recall that $m_{\ba}$ is defined as the unique critical point of the function $\phi_{\ba}$ from Proposition \ref{prop.mdag}, and so for any $x$, $\frac{d \left(\phi_{\ba}(m_{\ba}+tx)\right) }{dt}|_{t=0} = 0.$ Computing the derivative, we get
\[
0 = \phi_{\ba}(m_{\ba})\cdot\sum_{i\in I_{\ba}}\langle v_i, x\rangle\log(\langle v_i,m_{\ba}\rangle+a_i).
\]
Since $\phi_\ba(m_\ba)$ is positive, we have $\sum_{i\in I_{\ba}}\langle v_i, x\rangle\log(\langle v_i,m_{\ba}\rangle+a_i)=0$, and consequently 
\[
1= \prod_{i\in I_{\ba}}\left(\langle v_i,m_{\ba}\rangle+a_i\right)^{\langle v_i, x\rangle}.
\]
The other term is easy to estimate in a manner similar to the second product, i.e.
\[
 \left(\frac{1}{1+\frac{\langle v_i,x\rangle}{\sqrt{k}(\langle v_i,m_{\ba}\rangle+a_i)}} \right)^{\sqrt{k}\langle v_i,x\rangle}
= 
e^{-\frac{\langle v_i,x\rangle^2}{\langle v_i,m_{\ba}\rangle+a{\rho_i}}}\cdot(1+O(k^{3c-\frac{1}{2}})).
\]
Thus,
\begin{align}
\frac{\phi_{k\cdot \ba}(km_\ba)}{\phi_{k\cdot \ba}(km_\ba+\sqrt{k}x)}
\sqrt{\prod_{i\in I_{\ba}}\frac{ \langle v_i,km_{\ba} \rangle + k a_i  }{\langle v_i,km_{\ba} + \sqrt{k}x \rangle + k a_i  }}
= & 
1\cdot
(1+O(k^{c-\frac{1}{2}}))\\
\cdot &
e^{\frac{1}{2}\sum_{i\in I_{\ba}}\frac{\langle v_i,x\rangle^2}{\langle v_i,m_{\ba}\rangle+a{\rho_i}}}\cdot(1+O(k^{3c-\frac{1}{2}}))\\
\cdot &
e^{-\sum_{i\in I_{\ba}}\frac{\langle v_i,x\rangle^2}{\langle v_i,m_{\ba}\rangle+a{\rho_i}}}\cdot(1+O(k^{3c-\frac{1}{2}}))\\
= &
e^{-\frac{1}{2}\sum_{i\in I_{\ba}}\frac{\langle v_i,x\rangle^2}{\langle v_i,m_{\ba}\rangle+a{\rho_i}}}
\cdot
(1+O(k^{3c-\frac{1}{2}})).\\
\end{align}
We've assumed that $c<\frac{1}{6}$, so the error term goes to $0$.
\end{proof}

\section{The auxiliary distribution $\nu'_k$}

Because we are not able to directly access $\E[\textbf{X}_{k\cdot\ba}]$, it is easier to calculate the limit of an auxiliary distribution rather than the distributions of $\textbf{Y}_k$ (the $\nu_k$). Define
\[
\nu'_k : = \left(\tau_{\sqrt k}\right)_*\left( \mu_{k\cdot\ba}\ast \delta_{-km_{\ba}}\right).
\] 
From the definitions, $\nu'_k$ is a sum of Dirac measures with multinomial coefficients over lattice points in $kP$, but the next proposition shows that as $k$ goes to infinity, there is an alternative formula.

%
%
\begin{proposition}
Let $0<c<\frac{1}{6}$. There are constants $d_k$ such that
\[
\lim_{k\rightarrow\infty}\nu'_k
 = 
 \lim_{k\rightarrow\infty} \frac{1}{d_k}\sum_{x\in B_{k^c}(0)\cap \frac{M-km_\ba}{\sqrt k}} e^{-\frac{1}{2}\sum_{i\in I_\ba}\frac{\langle v_i, x\rangle^2}{\langle v_i,m_\ba\rangle+a_i}}
 \delta_{x}.
\]
\end{proposition}
\begin{proof}
By definition,
\begin{equation}
\nu'_k
=
 \frac{1}{b_k}\sum_{x\in kP\cap \Z^n}\ \binom{k\sum_i a_i}{\langle v_{1},x\rangle + k a_{1},\ldots, \langle v_{r},x\rangle + k a_{r}}  \delta_{\frac{x-km_{\ba}}{\sqrt k}},
\end{equation}
for $b_k$ normalizing constants (explicitly, $b_k = \sum_{x\in kP\cap \Z^n}\ \binom{k\sum_i a_i}{\langle v_{1},x\rangle + k a_{1},\ldots, \langle v_{r},x\rangle + k a_{r}}$). We will start by showing that most of the terms can be ignored as as $k$ goes to infinity. 

Let 
\[
c_{L} = \sup_{x\in L}\inf_{m\in L\cap\Z^n}|x-m|,
\]
and let $m_{\textrm{inner}, k}$ be a point in $B_{c_{L}+1}(km_{\ba})\cap \left(km_\ba+L\right)\cap\Z^n$ (which certainly must be non-empty).

The mass from terms in the sum outside of $B_{k^{\frac{1}{2}+c}}(km_{\ba})$ can be bounded:
\[
\frac{1}{b_k}\sum_{x\in \left(kP\setminus B_{k^{\frac{1}{2}+c}}(km_{\ba})\right)\cap \Z^n}\binom{k\sum_i a_i}{\langle v_{1},x\rangle + k a_{1},\ldots,\langle v_{r},x \rangle + k a_i}
\leq
\sum_{x\in \left(kP\setminus B_{k^{\frac{1}{2}+c}}(km_{\ba})\right)\cap \Z^n}
\frac
{\binom{k\sum_i a_i}{\langle v_{1},x\rangle + k a_{1},\ldots,\langle v_{r},x \rangle + k a_i}}
{\binom{k\sum_i a_i}{\langle v_{1},m_{\textrm{inner},k}\rangle + k a_{1},\ldots,\langle v_{r},m_{\textrm{inner},k} \rangle + k a_i}}.
\]

Using Stirling's formula (see e.g. \cite[Section 12.33]{ww}), we can bound the summands:
\begin{align}
\frac
{\binom{k\sum_i a_i}{\langle v_{1},x\rangle + k a_{1},\ldots,\langle v_{r},x \rangle + k a_i}}
{\binom{k\sum_i a_i}{\langle v_{1},m_{\textrm{inner},k}\rangle + k a_{1},\ldots,\langle v_{r},m_{\textrm{inner},k} \rangle + k a_i}}.
&
\leq c'\cdot \frac{\displaystyle\prod_{\substack{i\\ \langle v_i, m_{\textrm{inner},k}\rangle + k a_i \geq 1 }} \left(\langle v_i, m_{\textrm{inner},k}\rangle + k a_i\right)^{\frac{1}{2}+\langle v_i, m_{\textrm{inner},k}\rangle + k a_i}}{\displaystyle\prod_{\substack{i\\ \langle v_i, x\rangle + k a_i \geq 1 }}  \left(\langle v_i, x\rangle + k a_i\right)^{\frac{1}{2}+\langle v_i, x\rangle + k a_i}},\\
&
\leq c'' k^{l}\cdot \left(\frac{\phi_{\ba}\left(\frac{m_{\textrm{inner},k}}{k}\right)}{\phi_{\ba}\left(\frac{x}{k}\right)}\right)^k,
\end{align}
where $c', c'',$ and $l$ are constants independent of $x$. 

By assumption, $|\frac{ m_{\textrm{inner},k} }{k}-m_\ba |< \frac{c_{L}+1}{k}$. If $x\in kP\smallsetminus B_{k^{\frac{1}{2}+c}}(km_\ba),$ then $|\frac{x}{k}-m_\ba |\geq\frac{1}{k^{\frac{1}{2}-c}}$. Let us define distinguished points $m_{\textrm{outer},k}$ such that $|\frac{m_{\textrm{outer},k}}{k}-m_\ba| = \frac{1}{k^{\frac{1}{2}-c}}$ and additionally among all $x$ satisfying $|\frac{x}{k}-m_\ba| = \frac{1}{k^{\frac{1}{2}-c}}$, the point $m_{\textrm{outer},k}$ minimizes $x\mapsto \phi_\ba(\frac{x}{k})$.

Then, since $\phi_\ba$ is convex, for any $x$ such that $|\frac{x}{k}-m_\ba| \geq \frac{1}{k^{\frac{1}{2}-c}}$, we have $\phi_\ba(\frac{m_{\textrm{outer},k}}{k})\leq \phi_\ba(\frac{x}{k})$. On the other hand, by writing $\phi_\ba$ as a Taylor polynomial around $m_\ba$, we have that $\frac{\phi_\ba(\frac{m_{\textrm{inner},k}}{k})}{\phi_\ba(\frac{m_{\textrm{outer},k}}{k})}<R<1$ for all $k$ sufficiently large.

Thus, 
\begin{align}
\sum_{x\in \left(kP\setminus B_{k^{\frac{1}{2}+c}}(km_{\ba})\right)\cap \Z^n}
\frac
{\binom{k\sum_i a_i}{\langle v_{1},x\rangle + k a_{1},\ldots,\langle v_{r},x \rangle + k a_i}}
{\binom{k\sum_i a_i}{\langle v_{1},m_{\textrm{inner},k}\rangle + k a_{1},\ldots,\langle v_{r},m_{\textrm{inner},k} \rangle + k a_i}} 
< &
\sum_{x\in \left(kP\setminus B_{k^{\frac{1}{2}+c}}(km_{\ba})\right)\cap \Z^n} c'' k^{l}R^k\\
\leq &
p(k) c'' k^{l} R^k,
\end{align}
where $p(k)$ is a polynomial in $k$ (e.g. the Ehrhart polynomial of $P$ which counts integer points in $kP$). This quantity goes to $0$ as $k$ goes to infinity, showing that 
\[
\lim_{k\rightarrow\infty}\nu'_k=\lim_{k\rightarrow\infty} \frac{1}{b_k}\sum_{x\in B_{k^{\frac{1}{2}+c}}(km_\ba)\cap kP\cap \Z^n}\ \binom{k\sum_i a_i}{\langle v_{1},x\rangle + k a_{1},\ldots, \langle v_{r},x\rangle + k a_{r}}  \delta_{\frac{x-km_{\ba}}{\sqrt k}}.
\]

We can rewrite the terms within the limit on the right-hand side of the above as 
\[
\frac{1}{b_k}\sum_{x\in B_{k^{c}}(0)\cap L\cap \frac{\Z^n-km_\ba}{\sqrt{k}}}\ \binom{k\sum_i a_i}{\langle v_{1},km_\ba+\sqrt{k}x\rangle + k a_{1},\ldots, \langle v_{r},km_\ba+\sqrt{k}x\rangle + k a_{r}}  \delta_{x},
\]
so we can then start comparing. Note: to reduce the line sizes below, we write $\binom{k\sum_i a_i}
{\langle v_{i},km_\ba+\sqrt{k}x\rangle + k a_{i}}$ for $\binom{k\sum_i a_i}
{\langle v_{1},km_\ba+\sqrt{k}x\rangle + k a_{1},\ldots, \langle v_{r},km_\ba+\sqrt{k}x\rangle + k a_{r}}$.

Let $d_k = b_k\frac{\phi_{k\cdot\ba}(km_\ba)\sqrt{\prod_{i\in I_\ba}\langle v_i,k m_\ba\rangle+ka_i}}{(2\pi)^{1-|I_\ba|}(k\sum_{i\in I_\ba}a_i)^{\frac{1}{2}+k\sum_{i\in I_\ba}a_i}}$. Then we can compare $v'_k$ and $\frac{1}{d_k}\sum_{x\in B_{k^c}(0)\cap L \frac{\Z^n-km_\ba}{\sqrt k}} e^{-\frac{1}{2}\sum_{i\in I_\ba}\frac{\langle v_i, x\rangle^2}{\langle v_i,m_\ba\rangle+a_i}}
 \delta_{x}$ via

\begin{align*}
& \sum_{x\in B_{k^c}(0)\cap L \cap \frac{\Z^n-km_\ba}{\sqrt k}}\left|\frac{1}{b_k}\binom{k\sum_i a_i}
{\langle v_{i},km_\ba+\sqrt{k}x\rangle + k a_{i}} 
%
-
\frac{1}{d_k} e^{-\frac{1}{2}\sum_{i\in I_\ba}\frac{\langle v_i, x\rangle^2}{\langle v_i,m_\ba\rangle+a_i}}\right| \\
\leq & \sum_{x\in B_{k^c}(0)\cap L\cap \frac{\Z^n-km_\ba}{\sqrt k}}
\frac{1}{b_k} \left|\binom{k\sum_i a_i}
{\langle v_{i},km_\ba+\sqrt{k}x\rangle + k a_{i}} 
- 
\frac{(2\pi)^{1-|I_\ba|}  (k\sum_{i\in I_\ba}a_i)^{\frac{1}{2}+k\sum_{i\in I_\ba}a_i}}{\phi_{k\cdot\ba}(km_\ba+\sqrt{k}x)\sqrt{\prod_{i\in I_\ba}\langle v_i,k m_\ba+\sqrt{k}x\rangle+ka_i} }\right| \\
+ & 
\sum_{x\in B_{k^c}(0)\cap L\cap \frac{\Z^n-km_\ba}{\sqrt k}}
\left|
\frac{1}{b_k}\frac{(2\pi)^{1-|I_\ba|}  (k\sum_{i\in I_\ba}a_i)^{\frac{1}{2}+k\sum_{i\in I_\ba}a_i}}{\phi_{k\cdot\ba}(km_\ba+\sqrt{k}x)\sqrt{\prod_{i\in I_\ba}\langle v_i,k m_\ba+\sqrt{k}x\rangle+ka_i} }
-
\frac{1}{d_k} e^{-\frac{1}{2}\sum_{i\in I_\ba}\frac{\langle v_i, x\rangle^2}{\langle v_i,m_\ba\rangle+a_i}}\right|.
\end{align*}

As $k$ goes to infinity, the first sum goes to $0$ by applying standard bounds associated to Stirling's approximation. 

The terms of the second can be rewritten 
\[
\frac{(2\pi)^{1-|I_\ba|}  (k\sum_{i\in I_\ba}a_i)^{\frac{1}{2}+k\sum_{i\in I_\ba}a_i}}{b_k \phi_{k\cdot\ba}(km_\ba)\sqrt{\prod_{i\in I_\ba}\langle v_i,k m_\ba\rangle+ka_i}}
\left|
\frac{\phi_{k\cdot\ba}(km_\ba)\sqrt{\prod_{i\in I_\ba}\langle v_i,k m_\ba\rangle+ka_i}}{\phi_{k\cdot\ba}(km_\ba+\sqrt{k}x)\sqrt{\prod_{i\in I_\ba}\langle v_i,k m_\ba+\sqrt{k}x\rangle+ka_i}}
-
 e^{-\frac{1}{2}\sum_{i\in I_\ba}\frac{\langle v_i, x\rangle^2}{\langle v_i,m_\ba\rangle+a_i}}\right|.
\]

By Lemma \ref{lem.maintech}, the whole sum is therefore bounded by some constant times
\[
k^{3c-\frac{1}{2}}\frac{1}{b_k}\sum_{x\in B_{k^c}(0)\cap L\cap \frac{\Z^n-km_\ba}{\sqrt k}}
\frac{(2\pi)^{1-|I_\ba|}  (k\sum_{i\in I_\ba}a_i)^{\frac{1}{2}+k\sum_{i\in I_\ba}a_i}}{\phi_{k\cdot\ba}(km_\ba)\sqrt{\prod_{i\in I_\ba}\langle v_i,k m_\ba\rangle+ka_i}}e^{-\frac{1}{2}\sum_{i\in I_{\ba}} \frac{\langle v_i,x\rangle^2}{\langle v_i,m_{\ba}\rangle + a_i}}.
\]

Recall that $b_k = \sum_{y\in kP\cap \Z^n}\ \binom{k\sum_i a_i}{\langle v_{1},y\rangle + k a_{1},\ldots, \langle v_{r},y\rangle + k a_{r}}$, so using Stirling's formula and Lemma \ref{lem.maintech}, we have for some small $\epsilon>0$ that
\[
b_k\geq (1-\epsilon) \sum_{y\in B_{k^{c}}(0)\cap L\cap \frac{\Z^n-km_\ba}{\sqrt{k}}}\ \frac{(2\pi)^{1-|I_\ba|}  (k\sum_{i\in I_\ba}a_i)^{\frac{1}{2}+k\sum_{i\in I_\ba}a_i}}{\phi_{k\cdot\ba}(km_\ba)\sqrt{\prod_{i\in I_\ba}\langle v_i,k m_\ba\rangle+ka_i}}e^{-\frac{1}{2}\sum_{i\in I_{\ba}} \frac{\langle v_i,x\rangle^2}{\langle v_i,m_{\ba}\rangle + a_i}},
\]
when $k$ is large enough. Thus besides the $k^{3c-\frac{1}{2}}$ factor, the expression is bounded above by a constant, and so the whole sum vanishes. 
\end{proof}

\begin{corollary}
The measures $\nu'_k$ weakly converge to a scalar multiple of the measure 
\[
e^{-\frac{1}{2}\sum_{i\in I_\ba} \frac{\langle v_i,x\rangle^2}{\langle v_i,m_\ba\rangle+a_i}}\cdot\delta_L.
\]
\end{corollary}

This is true because the sequence
\[
\frac{1}{d_k}\sum_{x\in B_{k^c}(0)\cap L\cap \frac{\Z^n-km_\ba}{\sqrt k}} e^{-\frac{1}{2}\sum_{i\in I_\ba}\frac{\langle v_i, x\rangle^2}{\langle v_i,m_\ba\rangle+a_i}}
 f(x),
 \]
converges to some multiple of $\int_L e^{-\frac{1}{2}\sum_{i\in I_\ba} \frac{\langle v_i,x\rangle^2}{\langle v_i,m_\ba\rangle+a_i}}f(x)dx$ by the definition of the Riemann integral.

Since $\lim_{k\rightarrow\infty}\frac{1}{d_k}\sum_{x\in B_{k^c}(0)\cap L\cap \frac{\Z^n-km_\ba}{\sqrt k}} e^{-\frac{1}{2}\sum_{i\in I_\ba}\frac{\langle v_i, x\rangle^2}{\langle v_i,m_\ba\rangle+a_i}}\delta_x$ is equal to $\lim_{k\rightarrow\infty}\nu'_k$ which is a probability measure, we can determine finally that $\lim_{k\rightarrow\infty}\nu'_k = \frac{e^{-\frac{1}{2}\sum_{i\in I_\ba}\frac{\langle v_i, x\rangle^2}{\langle v_i,m_\ba\rangle+a_i}}\delta_L}{\int_L e^{-\frac{1}{2}\sum_{i\in I_\ba}\frac{\langle v_i, x\rangle^2}{\langle v_i,m_\ba\rangle+a_i}}dx}$.

\section{The limit of $\nu_k$}

Now we can directly address the distribution of $\textbf{Y}_k$, which is the rescaling of the polyhedral multinomial random vector $X_k$, translated to have mean $0$. The distribution has the formula
\[
\nu_k = \left(\tau_{\sqrt k}\right)_*\left( \mu_{k\cdot\ba}\ast \delta_{-\E[\textbf{X}_{k\cdot\ba}]}\right).
\]
According to the results of the previous section, 
\[
\lim_{k\rightarrow\infty}\nu'_k = \lim_{k\rightarrow\infty}\left(\tau_{\sqrt k}\right)_*\left( \mu_{k\cdot\ba}\ast \delta_{-k\cdot m_\ba}\right) = \frac{e^{-\frac{1}{2}\sum_{i\in I_\ba}\frac{\langle v_i, x\rangle^2}{\langle v_i,m_\ba\rangle+a_i}}\delta_L}{\int_L e^{-\frac{1}{2}\sum_{i\in I_\ba}\frac{\langle v_i, x\rangle^2}{\langle v_i,m_\ba\rangle+a_i}}dx}.
\]
\begin{theorem}
The limit distribution of $\textbf{Y}_k$ is the probability measure 
\[
\frac{e^{-\frac{1}{2}\sum_{i\in I_\ba}\frac{\langle v_i, x\rangle^2}{\langle v_i,m_\ba\rangle+a_i}}\delta_L}{\int_L e^{-\frac{1}{2}\sum_{i\in I_\ba}\frac{\langle v_i, x\rangle^2}{\langle v_i,m_\ba\rangle+a_i}}dx}
\]
\end{theorem}
\begin{proof}
Integrating $(x_1,\ldots,x_n)$ against $\lim_{k\rightarrow\infty}\nu'_k$ is $0$, so $\lim_{k\rightarrow\infty}\frac{\E[X_{k\cdot\ba}]-k\cdot m_\ba}{\sqrt{k}} =0$. Therefore 
\begin{align*}
\lim_{k\rightarrow\infty}\textbf{Y}_k  & = \lim_{k\rightarrow\infty} \frac{X_{k\cdot\ba}-\E[X_{k\cdot\ba}]}{\sqrt{k}}, \\
& = \lim_{k\rightarrow\infty} \frac{X_{k\cdot\ba}-km_\ba}{\sqrt{k}},
\end{align*}
which is the random variable with distribution $\lim_{k\rightarrow\infty} \nu'_k$.
\end{proof}


\end{document}